\documentclass{amsart}

\usepackage{amsmath,amssymb,amsthm,graphicx,slashed,subfig}%
\usepackage{amsmath}%
\usepackage{amsfonts}%
\usepackage{amssymb}%

\usepackage{hyperref,enumerate}

\pdfoutput=1

\usepackage{colortbl}

\usepackage[all,cmtip]{xy}

\usepackage{graphicx}
\providecommand{\U}[1]{\protect\rule{.1in}{.1in}}
\newtheorem{thm}{Theorem}

\newtheorem{lma}[thm]{Lemma}

\newtheorem{prop}[thm]{Proposition}
\newtheorem{defn}[thm]{Definition}

\newtheorem{rem}[thm]{Remark}

\def\A{\mathcal{A}}
\def\cS{\mathcal{S}}
\def\B{\mathcal{B}}

\newcommand{\bra}[1]{\langle #1 |}
\newcommand{\ket}[1]{| #1 \rangle}
\def\C{\mathbb{C}}

\newcommand{\CZ}[1]{C^*(\Z)_{(#1)}}

\def\E{\mathcal{E}}

\def\epsilon{\varepsilon}

\def\H{\mathcal{H}}

\def\phi{\varphi}
\def\R{\mathbb{R}}

\def\cS{\mathcal{S}}

\DeclareMathOperator{\tr}{Tr}
\newcommand{\Toep}[1]{C(S^1)^{(#1)}}
\def\U{\mathcal{U}}
\def\Z{\mathbb{Z}}

\newtheoremstyle{commentstyle}
  {0.2cm}{0.2cm}
  {\sf}
  {0cm}
  {\bfseries}{ }
  {0cm}
  {\thmname{#1}\thmnumber{ #2}:\thmnote{ #3}}

\theoremstyle{commentstyle}
\newtheorem{mycomment}{Comment}

\title[GH convergence of state spaces for spectral truncations]{Gromov--Hausdorff convergence of state spaces for spectral truncations}

\author{Walter D. van Suijlekom}

\date{January 12, 2021}

\address{Institute for Mathematics, Astrophysics and Particle Physics, Radboud
University Nijmegen, Heyendaalseweg 135, 6525 AJ Nijmegen, The Netherlands.
}
\email{waltervs@math.ru.nl}

\begin{document}
\maketitle

\begin{abstract}
  We study the convergence aspects of the metric on spectral truncations of geometry. We find general conditions on sequences of operator system spectral triples that allows one to prove a result on Gromov--Hausdorff convergence of the corresponding state spaces when equipped with Connes' distance formula. We exemplify this result for spectral truncations of the circle, Fourier series on the circle with a finite number of Fourier modes, and matrix algebras that converge to the sphere. 
  \end{abstract}

\section{Introduction}
\label{sect:intro}


We continue our study of spectral truncations of (noncommutative) geometry that we started in \cite{CS20} and here focus on the metric convergence aspect of so-called operator system spectral triples.
This is part of a program that tries to extend the spectral approach to geometry to cases where (possibly) only part of the spectral data is available, very much in line with \cite{ALM14}. And even though the mathematical motivation should be sufficient, there is a clear physical motivation for this. Indeed, from experiments we will only have access to part of the spectrum since we are limited by the power and resolution of our detectors: we typically study physical phenomena up to a certain energy scale and with finite resolution.

The usual spectral approach to geometry \cite{C94} in terms of a $*$-algebra $\A$ of operators on $\H$ and a self-adjoint operator $D$ on $\H$ has been adapted in \cite{ALM14,CS20} to deal with such spectral truncations. The $*$-algebra is replaced by an {\em operator system} $E$ (dating back to \cite{CE77}), which is by definition a $*$-closed subspace of $\B(\H)$  containing the identity.

More precisely, we have the following definition. 
\begin{defn} An {\em operator system spectral triple} is a triple $(\E,\H,D)$ where $\E$ is a dense subspace of an operator system $E$ in $\mathcal B(\H)$, $\H$ is a Hilbert space and $D$ is a self-adjoint operator in $\H$ with compact resolvent and such that $[D,T]$ is a bounded operator for all $T \in \E$.
\end{defn}
An operator system comes with an {\em ordering}, namely, one can speak of positive operators in $E \subseteq \mathcal B(\H)$. As a consequence states on $E$ can be defined as positive linear functionals of norm 1. The above triple then induces a (generalized) distance function on the state space $\cS(E)$ by setting
\begin{equation}
\label{eq:dist}
d(\phi,\psi) = \sup_{x \in \E} \left\{ | \phi(x) - \psi(x)|: \|x\|_1 \leq 1 \right\}
\end{equation}
where $\|\cdot \|_1$ denote the {\em Lipschitz semi-norm}:
$$
\| x \|_1 = \| [D,x]\|; \qquad (x \in\E).
$$

If $\E = \A$ is a $*$-algebra then this reduces to the usual distance function \cite{C89,C94} on the state space of the $C^*$-algebra $A = \overline \A$ . It also agrees with the definition of quantum metric spaces based on order-unit spaces given in  \cite{Rie99,Rie00,Ker03,KL09,Lat15b,Lat16,Lat16b}. Note, however, that in the present work we restrict our attention to the metric structure on the state spaces, that is to say, as an ordinary metric space. In contrast, in {\em loc.cit.} the authors develop the notion of quantum metric space and quantum Gromov--Hausdorff distance which are formulated in the dual category (of $C^*$-algebras, order-unit spaces, {\em et cetera}) and with a more general version of the above distance function.

\bigskip

So, we will study the properties of this metric distance function and the notions of Gromov--Hausdorff convergence it gives rise to. We consider sequences of spectral triples on operator systems and formulate general conditions under which we prove the state spaces equipped with the above distance functions to converge to a limiting state space. The latter is also described by an operator system spectral triple. One of the novelties of our work is that we use the ideas of correspondences between compact metric spaces and their relation to Gromov--Hausdorff convergence as described for instance in \cite[Section 7.3.3]{BBI01}.

We exemplify our main result on Gromov--Hausdorff convergence by considering:
\begin{itemize}
\item spectral truncations on the circle;
\item Fourier series with only a finite number of non-zero Fourier coefficients;
  \item matrix algebras converging to the sphere.
  \end{itemize}
Previous results in the literature on the distance function for spectral truncations have been reported in \cite{ALM14,GS19a,GS19b}. However, in these works the distance function on states of the truncated system was only computed after pulling back these states to the original metric geometry. Extensions of the results contained in the present paper to tori are contained in the master's thesis \cite{Ber19}.

The convergence of matrix algebras to the sphere was studied by Rieffel in \cite{Rie02} while computer simulations were performed in \cite{BG16}. Using the general approach below we re-establish  part of this convergence result, namely, the Gromov--Hausdorff convergence of the corresponding (classical) metric spaces.

We note that other convergence results on the distance function on quantum spaces are obtained for quantum tori in \cite{Lat15}, for coherent states on the Moyal plane in \cite{ALV13}. More generally, in \cite{GS19b} certain sets of states have been identified for which the Connes' distance formula has good convergence properties with respect to a given metric on a Riemannian manifold.

\subsubsection*{Acknowledgements}
I would like to thank IHÉS for their hospitality and support during a visit in February 2020. I thank Alain Connes, Jens Kaad and Marc Rieffel for fruitful discussions. I am grateful to an anonymous referee for useful remarks and suggestions.

\section{Gromov--Hausdorff convergence for operator systems}
\label{sect:gh}
Given a sequence of operator system spectral triples $(\E_n,\H_n,D_n)$ we want to understand when and how this approximates an operator system spectral triple $(\E,\H,D)$. We will adopt the point of view of \cite{Rie00} and consider the convergence (in Gromov--Hausdorff distance) of the corresponding state spaces $\cS(E_n) \to \cS(E)$ equipped with the distance formula \eqref{eq:dist}. 
Since this notion of convergence is most suited to deal with {\em compact} metric spaces, we will assume below ({\em cf.} Theorem \ref{thm:GH-conv}) that the topology defined by the metric $d$ coincides with the weak-$*$ topology (with respect to which we know the state spaces to be compact). For the examples that follow this assumption is indeed satisfied, see also Remark \ref{rem:topologies} below.

\begin{defn}
  \label{defn:GH-conv}
Let $\{ (\E_n,\H_n,D_n)\}_n$ be a sequence of operator system spectral triples and let $(\E,\H,D)$ be an operator system spectral triple. An {\em approximate order isomorphism} for this set of data is given by linear maps $R_n: E \to E_n$ and $S_n: E_n \to E$ for any $n$ such that the following three condition hold:
\begin{enumerate}
\item the maps $R_n,S_n$ are positive,  unital  maps
\item there exist sequences $\gamma_n, \gamma_n'$ both converging to zero such that
\begin{align*}
\| S_n \circ R_n (a) - a\| &\leq \gamma_n \| a \|_1,\\
\| R_n \circ S_n (h) - h\| &\leq \gamma_n' \|h \|_1.
\end{align*}
\end{enumerate}
\end{defn}
In other words, we use the Lipschitz semi-norms to quantify how close the positive maps $R_n$ and $S_n$ are to being each other's inverse ({\it i.e.} form an order isomorphism) as $n\to \infty$. 

We will call a map between operator systems {\em $C^1$-contractive} if it is contractive with respect to both the operator norms and the Lipschitz semi-norms (thus assuming that we are given two operator system spectral triples for them). Finally, we say that the pair of maps $(R_n,S_n)$ is a {\em $C^1$-approximate order isomorphism} if $(R_n,S_n)$ is an approximate order isomorphism in the above sense and for which all maps $R_n$ and $S_n$ are $C^1$-contractive.

Note that the positivity and unitality condition on $R_n, S_n$ in particular implies that we may pull back states as follows:
\begin{align*}
&R_n^* : \cS(E_n) \to \cS(E); \qquad \phi_n \mapsto  \phi_n \circ R_n,\\
&S_n^* : \cS(E) \to \cS(E_n); \qquad \phi\mapsto \phi \circ S_n.
\end{align*}

\begin{rem}
  \label{rem:topologies}
  Even though it would be more natural to consider completely positive maps $R_n,S_n$ between the operators systems $E_n$ and $E$, this turns out not to be necessary for the proof of our main result as in fact we restrict our attention to the states space and the metric thereon. However, in all examples discussed below we find that $E$ is a commutative $C^*$-algebra so that these maps are in fact completely positive ({\em cf.} \cite[Theorems 3.9 and 3.11]{Pau02}).  
  \end{rem}

Let us denote the distance functions \eqref{eq:dist} for $(\E_n,\H_n,D_n)$ and $(\E,\H,D)$ by $d_{E_n}$ and $d_{E}$, respectively.

\begin{prop}
  \label{prop:estimates-dist}
If $(R_n,S_n)$ is a $C^1$-approximate order isomorphism for $(\E_n,\H_n,D_n)$ and $(\E,\H,D)$, then 
\begin{enumerate}
\item For all $\phi_n,\psi_n \in \cS(E_n)$ we have
$$
d_E(\phi_n \circ R_n, \psi_n \circ R_n) \leq d_{E_n} (\phi_n,\psi_n) \leq d_E(\phi_n\circ R_n, \psi_n \circ R_n) +2 \gamma_n'.
$$
\item For all $\phi,\psi \in \cS(E)$ we have
$$
d_{E_n}(\phi \circ S_n, \psi \circ S_n) \leq d_E (\phi,\psi) \leq d_{E_n}(\phi \circ S_n, \psi \circ S_n) +2 \gamma_n.
$$
\end{enumerate}
\end{prop}
\proof
Since $R_n$ is Lipschitz contractive it follows that if $\| a \|_1 \leq 1$ then also $\|R_n(a)\|_1\leq 1$. Hence
$$
\sup_{a \in \E} \left\{ | \phi\circ R_n(a) - \psi\circ R_n(a)|: \| a\|_1 \leq 1 \right\} \leq \sup_{h \in \E_n}  \left\{ | \phi (h) - \psi(h)|: \| h\|_1 \leq 1 \right\}.
$$
This establishes the first inequality (also proven in \cite[Proposition 3.6]{ALM14}).

For the second, note that for all $h \in \E_n$ with $\|h \|_1 \leq 1$ we have 
\begin{align*}
| \phi_n(h) -\psi_n(h) | &\leq |\phi_n(R_n(S_n(h))) - \psi_n(R_n(S_n(h))) | \\
&\quad + | \phi_n(h)- \phi_n(R_n(S_n(h)))| +| \psi_n(h) - \psi_n(R_n(S_n(h)))| \\
& \leq  d_E(\phi_n\circ R_n, \psi_n \circ R_n) + 2 \gamma_n'.
\end{align*}
since $\| \phi_n\| = \|\psi_n\|=1$ and $\|S_n(h)\|_1 \leq \| h \|_1 \leq 1$.
The second claim follows similarly. 
\endproof

The final justification for the above definition of $C^1$-approximate order isomorphism is our following, main result. 
\begin{thm}
  \label{thm:GH-conv}
  Suppose $(\E_n,\H_n,D_n)$ ($n=1,2,\cdots$) and $(\E,\H,D)$ are operator system spectral triples such that the topologies on $\cS(E_n)$ and $\cS(E)$ defined by the metrics $d_{E_n}$ and $d_E$, respectively, agree with the weak-$*$ topologies on them.

If $(R_n,S_n)$ is a $C^1$-approximate order isomorphism for $(\E_n,\H_n,D_n)$ and $(\E,\H,D)$,  then the state spaces $(\cS(E_n), d_{E_n})$ converge to $(\cS(E),d_E)$ in Gromov--Hausdorff distance. 
\end{thm}
\proof
 
This follows by applying techniques for correspondences between metric spaces and their relation to Gromov--Hausdorff convergence via the notion of distortion ({\em cf.} \cite[Theorem 7.3.25]{BBI01}). In the case at hand one has for the Gromov--Hausdorff distance that 
$$
d_{\textup{GH}} \left ((\cS(E_n),d_{E_n}), (\cS(E),d_{E}) \right) \leq  \frac 12 \textup{dis} (\mathfrak R_n) 
  $$
  where the correspondence $\mathfrak R_n \subset \cS(E_n)\times S(E)$ is defined by
  $$
\mathfrak R_n = \left\{ ( \phi_n, R_n^* (\phi_n) ) : \phi_n \in \cS(E_n) \right\} \cap \left\{ ( S_n(\phi), \phi ) : \phi \in \cS(E) \right\}
  $$
with distortion ({\em cf.} \cite[Defn 7.3.21]{BBI01})
$$
\textup{dis}(\mathfrak R_n) = \sup\left \{ \left| d_{E_n}(\phi_n , \phi_n') - d_E (\phi, \phi') \right| : (\phi_n,  \phi) , ( \phi_n',  \phi')  \in \mathfrak R_n   \right\}
$$
But from Proposition \ref{prop:estimates-dist} it follows that this can be bounded by $2 \gamma_n + 2 \gamma_n'$, which converges to 0 as $n \to \infty$. 
\endproof

\begin{rem}
  Given an operator system spectral triple $(E,\H,D)$, it is an interesting question to see when the metric topology on $\cS(E)$ defined by $d$ coincides with the weak-$*$ topology. For the commutative case, this was already established in \cite{C89} while more generally it is also shown in that paper that if the set $\{ h \in E:  \| [D,h ] \| \leq 1 \}/\C 1$ is bounded in $E$, then $d$ is a metric. Rieffel then established \cite[Theorem 1.8]{Rie98} the most general result (for the compact case) stating that when the above set is totally bounded, the $d$-topology agrees with the weak-$*$ topology. 

    Below we will consider only finite-dimensional operator system spectral triples for which also $\ker [D,\cdot ] = \C$, so that the above set is totally bounded. Indeed, in this case $\| [D,\cdot ]\|$ induces a norm on the quotient $E/ \C 1$ for which the unit ball is compact. Hence, under these assumptions the $d$-topology coincides with the weak$-*$ topology; a fact that we will tacitly use throughout the rest of the paper. 
\end{rem}

\begin{rem}
  Following up on the previous remark, there is a close relation between our notion of $C^1$-approximate order isomorphism and the {\em $(\epsilon,C)$-approximations} of operator systems $E$ equipped with a Lipschitz semi-norm $L$ that were considered recently in \cite{Kaa20}. Such $(\epsilon,C)$-approximations are defined by two maps $\imath,\phi:E \to F$ into some other operator system $F$ such that (1) $C^{-1} \|a \| \leq \| \imath(a) \| \leq C \| a \|$, $\| \phi(a)\| \leq  C \| a \|$ for all $a \in E$; (2) $\text{ran} (\phi) \subseteq F$ is finite dimensional; and (3) $\| \imath(a) - \phi(a)\| \leq \epsilon L(a)$. 

  Thus, if $(R_n,S_n)$ is a $C^1$-approximate order isomorphism in the sense we just defined, then, with the additional assumption that $\text{ran}(R_n)$ or $\text{ran}(S_n)$ is finite-dimensional, the pair $(\imath=\text{id}_E,S_n \circ R_n)$ gives a $(\gamma_n ,1)$-approximation of the pair $(E,\| \cdot \|_1)$. Note that in all examples discussed below this assumption of finite-dimensionality will be met.
  \end{rem}

\section{Examples of Gromov--Hausdorff convergence}
\subsection{Spectral truncations of the circle converge}
We will analyze a spectral truncation of the distance function on the circle, the latter being described by the spectral triple
\begin{equation}
\left(\A = C^\infty(S^1), \H = L^2(S^1),D = - i \frac{d}{dx} \right).
\end{equation}
We will consider a spectral truncation defined by the orthogonal projection $P=P_{n}$ of rank $n$ onto span$_\C\{ e_{1},e_{2},\ldots , e_{n}\}$ for some fixed $n \geq  1$, where $e_k(x) = e^{i k x}$ ($k \in \Z$) is the orthonormal (Fourier) eigenbasis of $D$. In the following we will suppress the representation of $C(S^1)$ on $L^2(S^1)$ by pointwise multiplication and simply write $f$ for the corresponding bounded operator. 
An arbitrary element $T=PfP$ in $PC(S^1)P$ can be written as an $n\times n$ Toeplitz matrix $T_{kl} = \langle e_k ,f e_l \rangle = a_{k-l}$ in terms of the Fourier coefficients $a_j$ of $f$. In matrix form we thus have
\begin{equation}\label{toeplitz}
T= \begin{pmatrix} a_0 & a_{-1} & \cdots &a_{-n+2} &a_{-n+1} \\
    a_1 & a_0 & a_{-1} & & a_{-n+2}\\
    \vdots & a_1 &a_0 &\ddots & \vdots\\
a_{n-2} & & \ddots & \ddots & a_{-1} \\
a_{n-1} & a_{n-2} & \cdots & a_1 & a_0\\
\end{pmatrix}.
\end{equation}
The corresponding operator system $P C(S^1) P = P C^\infty(S^1) P$ is called the {\em Toeplitz
  operator system} and is denoted by $\Toep{n}$; it has been analyzed at
length in \cite{CS20}.

An operator system spectral triple for the Toeplitz operator system is given by $(\Toep{n} , P L^2(S^1),P D P)$. 
\subsubsection{Fej\'er kernel}
Clearly, the compression $f \mapsto PfP$ by $P=P_n$ defines a positive map $R_n: C(S^1)\to \Toep{n}$. In order to find an approximate inverse to this map we take inspiration from \cite[Section 2]{Rie02}. In fact, we define $S_n: \Toep{n} \to C(S^1)$ to be its (formal) adjoint when we equip $C(S^1)$ with the $L^2$-norm and $\Toep{n}$ with the (normalized) Hilbert--Schmidt norm. Let $\alpha_x$ denote the natural action of $S^1$ on $\Toep{n}$, and define a norm 1 vector $\ket \psi$ in $PL^2(S^1)$ by
$$
\ket \psi = \frac{1}{\sqrt{n}} \left ( e_{1} + \cdots + e_n \right).
$$

\begin{figure}
\centering
\includegraphics[scale=.6]{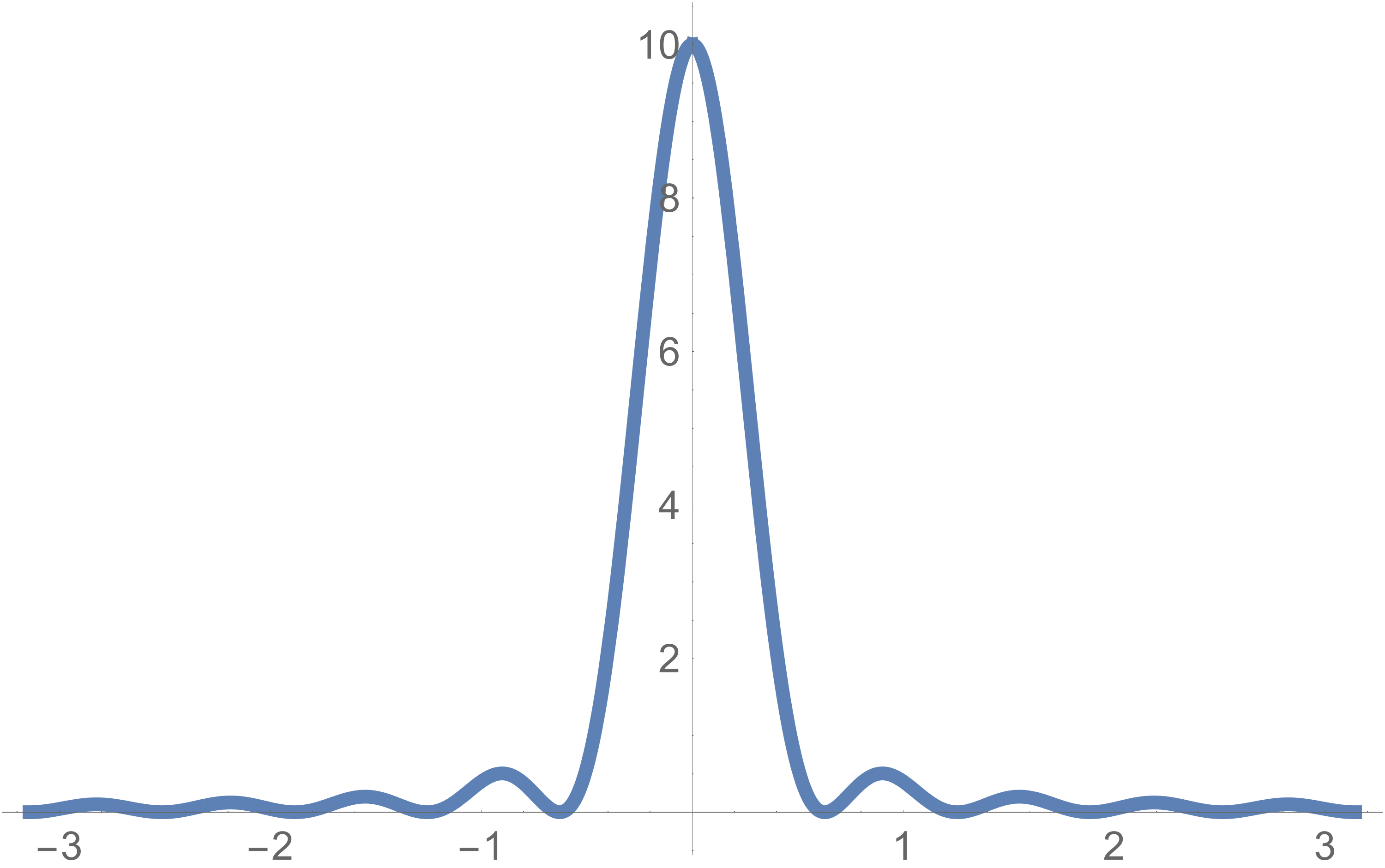}
\caption{The Fej\'er kernel $F_N = \frac 1N \frac{\sin^2(Nx/2)}{\sin^2(x/2)}$ for $N=10$.}
\end{figure}

\begin{prop}
\label{prop:adjoint-HS}
The map $S_n : \Toep{n} \to C(S^1)$ defined for any $T\in \Toep{n}$ by $S_n(T)(x) = \tr \left(\ket{\psi}\bra{\psi} \alpha_x(T)\right)$ satisfies 
$$
\langle f, S_n(T) \rangle_{L^2(S^1)} =\frac 1{n} \tr\left( (R_n(f))^* T \right).
$$
Moreover, we may write
$$
S_n(R_n(f))(x) = \sum_{k=-n+1}^{n-1} \left(1- \frac{|k|}{n} \right) a_k e^{ik x} = (F_{n}* f )(x)
$$
in terms of the Fej\'er kernel $F_{n}$ and the Fourier coefficients $a_k$ of $f$.
\end{prop}
\proof
Let us first check the formula for $S_n(T)$ by computing that
$$
\tr \left(\ket{\psi}\bra{\psi} \alpha_x(T)\right) = \frac 1 {n} \sum_{k,l} T_{kl} e^{i (k-l)x}= \frac 1 {n} \sum_{k=-n+1}^{n-1} (n- |k|) a_{k} e^{ikx} 
$$
Thus, $S_n(T) = F_{n} * f$ when $T=PfP$ ---again understanding $f$ as an operator acting on $L^2(S^1)$--- and we may use elementary Fourier theory  ({\em cf.} \cite[Proposition 3.1(vi)]{SS03})to derive
\begin{align*}
  \langle g, S_n(T) \rangle &= \langle g, F_{n} * f \rangle = \sum_{k=-n+1}^{n-1} \overline{b_k } a_k \left(1- \frac{|k|}{n} \right),
\intertext{where $b_k$ are the Fourier coefficients of $g$. On the other hand, we have}
\frac 1{n} \tr\left( (R_n(g))^* T \right) &= \frac 1{n} \sum_{k,l=-n+1}^{n-1} \overline{b_{k-l}} a_{k-l} = \frac 1 {n} \sum_{k=-n+1}^{n-1} \overline{b_k } a_k \left(n- |k| \right).
\qedhere
\end{align*}
\endproof

\subsubsection{The circle as a limit of its spectral truncations}
Let us now show in a series of Lemma's that the conditions of Definition \ref{defn:GH-conv} are satisfied.

\begin{lma}
\label{lma:est-comm-f}
For any $f \in C^\infty(S^1)$ we have $\| R_n(f) \| \leq \| f \|$ and $\| [D,R_n(f)]\| \leq \| [D,f]\|$. 
\end{lma}
\proof
Since $R_n(f) = P f P$ and $P$ commutes with $D$ this follows directly since $P$ is a projection.
\endproof

\begin{lma}
\label{lma:est-f}
There exists a sequence $\{\gamma_n\}$ converging to 0 such that
$$
\| f - S_n (R_n(f)) \| \leq \gamma_n \| [D,f]\| 
$$
for all $f \in C^\infty(S^1)$.
\end{lma}
\proof
Since $\|[D,f]\|$ is equal to the Lipschitz constant for $f$ \cite[Proposition 1]{C89} we have that $| f(x) - f(z) | \leq |x-z| \| [D,f]\|$. We use this in the following estimate
\begin{align*}
| f(x) - S_n(R_n(f)) (x) | &\leq \frac{1}{2 \pi} \int_{-\pi}^\pi F_{n}(y) | f(x) - f(y-x) | dy \\
&\leq \frac{1}{2 \pi} \int_{-\pi}^\pi F_{n}(y)  |y|  dy \cdot \| [D,f]\|=: \gamma_n \|[D,f]\|.
\end{align*}
Becasue the Fej\'er kernels form an approximate delta-function at $0$ ({\em cf.} \cite[Theorem 4.1 and Lemma 5.1]{SS03}) imply that $\gamma_n \to 0$.
\endproof

\begin{lma}
\label{lma:est-comm-T}
For any $T \in \Toep{n}$ we have $\| S_n(T) \| \leq \|T \|$ and $\| [D,S_n(T)] \| \leq \| [D,T]\|$.
\end{lma}
\proof First note that $[D,S_n(T)]$ is a function on $S^1$ (it is $i$ times the derivative of $S_n(T)$). Moreover, we have
\begin{align*}
|[D,S_n(T)](x)| &= |\tr \left( \ket \psi \bra\psi \alpha_x ([D,T]) \right)|
\leq \| \ket \psi \bra \psi \|_1 \| \alpha_x ( [D,T]) \| \leq \| [D,T]\|.
\end{align*}
Since this holds for any $x$, we may take the supremum to arrive at the desired inequality. The other inequality is even easier. 
\endproof

\begin{lma}
\label{lma:est-T}
There exists a sequence $\{\gamma_n'\}$ converging to 0 such that
$$
\| T - R_n (S_n(T)) \| \leq \gamma'_n \| [D,T]\| 
$$
for all $T \in \Toep{n}$.
\end{lma}
\begin{proof}
Write $T = PgP$ for $g = \sum_k b_k e^{ik x}$. Then the matrix coefficients of the Toeplitz matrix $T- R_n(S_n(T))$ are given by
\begin{align*}
(T_{kl} - R_n(S_n(T))_{kl}) &= \begin{pmatrix} b_{k-l}\end{pmatrix} - \begin{pmatrix}1-\frac{|k-l|}{n} b_{k-l} \end{pmatrix} =  \begin{pmatrix}\frac{|k-l|}{n} b_{k-l} \end{pmatrix} \\
&=  (T_n - T_n^*) \odot \begin{pmatrix} \frac{k-l}{n} b_{k-l} \end{pmatrix} \\
& =\frac{1}{n}  (T_n - T_n^*) \odot \begin{pmatrix}[D,T] \end{pmatrix} 
\end{align*}
in terms of the Schur product $\odot$ with $T_n$ and $T_n^*$ where 
$$
T_n =\begin{pmatrix} 1 & 0 & \cdots &0 \\
1 & 1 & \cdots & 0 \\
\vdots & \vdots & \ddots & \vdots\\
1 & 1 & \cdots & 1 
\end{pmatrix}.
$$
Now the norm of the map $A \mapsto T_n \odot A$ for $A \in M_{n}(\C)$ coincides with $\| T_n \|_{cb}$ (cf. \cite[Chapter 8]{Pau02}). In \cite[Theorem 1]{ACN92} the following estimate for this norm was derived: 
$$
\| T_n \|_{cb} \leq  \left ( 1+ \frac 1 \pi ( 1 + \log (n))\right).
$$
Hence we have
\begin{align*}
\| T- R_n(S_n(T)) \|& \leq \frac{2}{n} \| T_n \|_{cb} \| [D,T] \| \leq \gamma_n' \|[D,T]\|
\end{align*}
where $\gamma_n' := \frac{2}{n} \left ( 1+ \frac 1 \pi ( 1 + \log (n))\right)$. It is clear that $\gamma_n' \to 0$ as $n \to \infty$.
\end{proof}

Thus we find that the pair of maps $(R_n,S_n)$ for $\{ (\Toep{n}, P_n L^2(S^1), P_n D P_n)\}_n $ and $(C^\infty(S^1), L^2(S^1), D)$ forms a $C^1$-approximate order isomorphism. We may conclude from Theorem \ref{thm:GH-conv} that
\begin{prop}
  \label{prop:gh-toep}
The sequence of state spaces $\{ (\cS(\Toep{n} ),d_n)\}_n$ converges to $(\cS(C(S^1)),d)$ in Gromov--Hausdorff distance. 
\end{prop}
Using a simple Python script we have computed the distance function for states on $\Toep{n}$ of the form $S_n^* ({\rm ev}_x)$ for $n=3,5,9$, where ${\rm ev}_x$ is the pure state on $C(S^1)$ given by evaluation at $x$. The optimization problem for computing the distance has been solved numerically using the standard {\em sequential least squares programming} (SLSQP) method and we claim absolutely no originality or proficiency here. We have illustrated the numerical results in Figure \ref{fig:dist-toep}.

\begin{figure}
  \includegraphics[scale=.6]{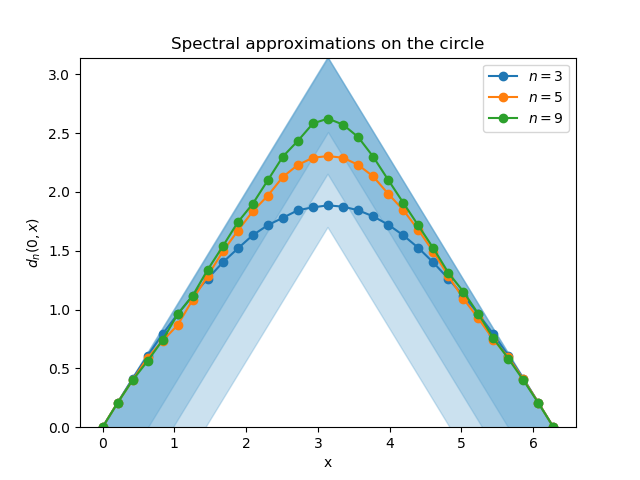}
  \caption{
    The distance function $d_n(0,x) \equiv d_n(0,S_n^*({\rm ev}_x))$ on the Toeplitz operator system (Proposition \ref{prop:gh-toep}) for $n=3,5,9$. The blue band corresponds to the lower bounds $d(0,x) - 2 \gamma_n$ given in Proposition \ref{prop:estimates-dist} with the constants $\gamma_n$ given in Lemma \ref{lma:est-f}.}
  \label{fig:dist-toep}
\end{figure}

\subsection{Fej\'er--Riesz operator systems converge to the circle}

In \cite{CS20} we found the dual operator system of $\Toep{n}$ to be equal to the operator system of functions on $S^1$ with only a finite number of non-zero Fourier coefficients. It gives a different type of truncation, this time taking place at the level of the function algebra, as opposed to a spectral truncation in Hilbert space. 

More precisely, we will consider the so-called {\em Fej\'er--Riesz operator system}:
\begin{equation}
  \label{eq:CstZn}
  \CZ{n} = \left\{ a = (a_k)_{k \in \Z}: \text{supp}( a) \subset (-n,n) \right\}.
\end{equation}
The elements in $\CZ{n}$ are thus given by sequences with finite support of the form
$$
a = (\ldots, 0 ,a_{-n+1},a_{-n+2},\ldots, a_{-1}, a_0, a_1, \ldots, a_{n-2},a_{n-1},0,\ldots)
$$
and this allows to view $\CZ{n}$ as an operator subsystem of $C^*(\Z) \cong C(S^1)$.

The adjoint $a \mapsto a^*$ is given by $a^*_k = \overline a_{-k}$ 
and an element $a \in \CZ{n}$ is positive iff $\sum_k a_k e^{ikx}$ defines a positive function on $S^1$.

Since this naturally is an operator subsystem of $C(S^1)$ it is natural to consider the following spectral triple: 

\begin{equation}
\left(\CZ{n} , \H = L^2(S^1),D = - i \frac{d}{dx} \right).
\end{equation}

We will be looking for positive and contractive maps $K_n: C(S^1) \to \CZ{n}$ and $L_n: \CZ{n} \to C(S^1)$ satisfying the conditions of Definition \ref{defn:GH-conv} so that we can apply Theorem \ref{thm:GH-conv} to conclude Gromov--Hausdorff convergence of the corresponding state spaces.

We introduce
\begin{align*}
  K_n : C(S^1) &\to \CZ{n}\\
  f &\mapsto F_n \ast f 
  \intertext{where we recall that $F_n = \sum_{|k| \leq n-1} (1-|k|/n ) e^{ikx}$ is the Fej\'er kernel so that $K_n$ indeed maps to $\CZ{n}$ considered as an operator subsystem of $C(S^1)$. The map $L_n$ is simply the linear embedding of $\CZ{n}$ as an operator subsystem of $C^*(\Z) \cong C(S^1)$:
}
  L_n : \CZ{n} &\to C(S^1)\\
  (a_k) &\mapsto \left( x \mapsto \sum_k a_k e^{ik x} \right).
\end{align*}
Positivity and contractiveness of $K_n$ for the norm and Lipschitz norm is an easy consequence of the good kernel properties of $F_n$ while for $L_n$ they are trivially satisfied. 

\begin{lma}
  \label{lma:est-f-FR}
  There exists a sequence $\gamma_n$ converging to 0 such that
  $$
\| L_n \circ K_n(f) - f \| \leq \gamma_n \| [D,f]\|
$$
for all $f \in C^\infty(S^1)$. 
  \end{lma}
\proof
Since $L_n \circ K_n(f) = F_n \ast f$ the proof is analogous to that of Lemma \ref{lma:est-f}.
\endproof

\begin{lma}
  There exists a sequence $\gamma_n'$ converging to 0 such that
  $$
\| K_n \circ L_n(a) - a \| \leq \gamma_n' \| [D,a]\|
$$
for all $a \in \CZ{n}$.
  \end{lma}
\proof
From the Fourier coefficients of the Fej\'er kernel we find that
$$
K_n \circ L_n(a) -a = \left(- \frac{|k|}{n} a_k \right)_k.
$$
We will estimate the sup-norm of the function $f(x) = \frac 1n \sum_k |k|a_k  e^{ik x}$ by the Lipschitz norm of $a$. First of all, we may write $f$ as a convolution product $f = g \ast h$ where $g = \sum_{k=-n+1}^{n-1} \text{sgn}(k) e^{ik x}$ and $h = \frac 1n \sum_{k=-n+1}^{n-1} k a_k  e^{ik x} = \frac 1 n [D,a]$. Then $\| f \|_\infty \leq \| g \|_1 \| h \|_\infty$ where
$$
\| g\|_1 \leq \| g \|_2 = \sqrt{ 2n-1} .
$$
We conclude that $\| g \ast h \|_\infty \leq \gamma'_n \| [D,a]\|_\infty$ with $\gamma_n' = \frac{\sqrt{2n-1}}{n} \to 0$ as $n \to \infty$. 
\endproof

We conclude that the pair of maps $(K_n,L_n)$ for $\{ (\CZ{n},L^2(S^1),D)\}_n$ and $(C^\infty(S^1),L^2(S^1),D)$ forms a $C^1$-approximate order isomorphism and we have 
\begin{prop}
  \label{prop:gh-FR}
The sequence of state spaces $\{ (\cS(\CZ{n}),d_n)\}_n$ converges to $(\cS(C(S^1)),d)$ in Gromov--Hausdorff distance. 
\end{prop}

We again illustrate the numerical results for the first few cases in Figure \ref{fig:dist-FR}. As compared to the Toeplitz operator system (Figure \ref{fig:dist-toep}) the optimization is much more cumbersome. This is essentially due to the fact that it involves the computation of a supremum norm of a trigonometric polynomial.

\begin{rem}
If we recall the duality between $\Toep{n}$ and $\CZ{n}$ as operator systems from \cite{CS20} it is quite surprising that both operator system spectral triples converge to the circle as $n \to \infty$.
\end{rem}

\begin{figure}
  \includegraphics[scale=.6]{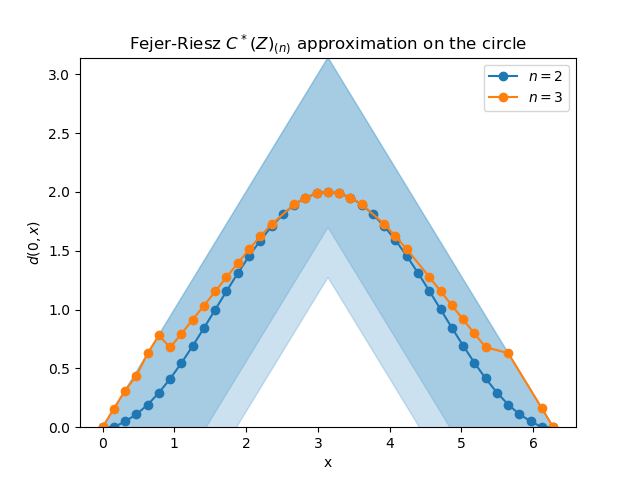}
  \caption{
    The distance function $d_n(0,x) \equiv d_n(0,L_n^* {\rm ev}_x)$ on the Fej\'er--Riesz operator system (Proposition \ref{prop:gh-FR}) for $n=2,3$. The blue band corresponds to the lower bounds $d(0,x) - 2 \gamma_n$ given in Proposition \ref{prop:estimates-dist} with the constants $\gamma_N$ given in Lemma \ref{lma:est-f-FR}.}
  \label{fig:dist-FR}
\end{figure}

\subsection{Matrix algebras converge to the sphere}
In \cite{Rie00,Rie02} Rieffel analyzed quantum Gromov--Hausdorff convergence for so-called quantum metric spaces. Such a space is given by a pair $(A,L)$ of an order-unit space $A$ and a so-called Lipschitz norm $L$ on $A$. At first sight, such spaces appear to be more general than (operator system) spectral triples and the distance function they give rise to. However, as Rieffel shows in \cite[Appendix 2]{Rie00} Dirac operators are universal in the sense that the Lipschitz semi-norms can always be realized as norms of commutators with a self-adjoint operator $D$.  Note that it remains an open question, however, for what Lipschitz semi-norms one can find an operator $D$ with compact resolvent implementing that semi-norm.

We show below that the corresponding state spaces with Connes' distance formula converge in Gromov--Hausdorff distance to the state space on the round two-sphere. This is closely connected ---certainly at the technical level--- to the  results of \cite{Rie02} which is that the matrix algebras that describe the fuzzy two-sphere converge in quantum Gromov--Hausdorff distance to the round two-sphere. Even though for much of the analysis we may refer to \cite{Rie00,Rie02} we do formulate the main results in our framework of operator system spectral triples and, as said, restrict our attention to the classical metric spaces.

\bigskip

We will describe the round two-sphere by the following spectral triple:
\begin{equation}
\label{eq:spectral-triple-S2}
(C^\infty(S^2),\C^2 \otimes L^2(S^2), D_{S^2})
\end{equation}
We write $S^2= \{ (x_1,x_2,x_3) \in \R^3: x_1^2 + x_2^2+x_3^2 = 1\}$ so that the following vector fields
$$
X_{jk} = x_j \partial_k - x_k \partial_j ; \qquad (j<k). 
$$
are tangent to $S^2$. Of course, these vector fields are fundamental vector fields and generate the Lie algebra $su(2)$. Note that the normal vector field is given by $\vec{x}$ itself. 

In terms of the three Pauli matrices we may then write the Dirac operator as \cite{Tra92}
\begin{equation}
\label{eq:dirac-S2}
D_{S^2} = (\vec x \cdot \vec \sigma) \sum_{j<k} \sigma^j \sigma^k \otimes X_{jk}. 
\end{equation}
Note that $(\vec x \cdot \vec \sigma) := \sum_{k=1}^3 x_k \sigma_k$ acts as the chirality operator and makes sure that the spinor bundle on $S^2$ is actually non-trivial (as it should).

The {\em fuzzy sphere} \cite{Mad92} is obtained when one considers spherical harmonics on the sphere only up to some maximum total spin. More precisely, it is described by a matrix algebra $L(V_n)$ where $V_n$ is the $n$-dimensional irreducible representation of $SU(2)$. A Dirac operator on the fuzzy sphere was introduced in \cite{GP95} (see also \cite{Bar15}):
\begin{equation}
\label{eq:dirac-fuzzy-S2}
D_n := \sum_{j<k} \sigma^j \sigma^k \otimes [L_{jk}, \cdot ]
\end{equation}
where $L_{jk}$ are standard generators of $su(2)$ in the $n$-dimensional representation, satisfying
$$
[L_{jk},L_{lm} ] = \delta_{kl} L_{jm} - \delta_{km} L_{jl} - \delta_{jl} L_{km} + \delta_{jm} L_{kl}.
$$
This  gives rise to the following spectral triple
\begin{equation}
\label{eq:spectral-triple-fuzzy-S2}
(L(V_n), \C^2 \otimes L(V_n) , D_n)
\end{equation}
The comparison between \eqref{eq:dirac-S2} and \eqref{eq:dirac-fuzzy-S2} is convincing, except for the absence of the chirality operator in the case of the fuzzy sphere. However, as shown in \cite{Bar15} this can be repaired for by a doubling of the representation space and a corresponding doubling constructing for the Dirac operator. Note that this does not alter the corrresponding Lipschitz norms, so we may just as well work with the Dirac operator defined in \eqref{eq:dirac-fuzzy-S2}.

\begin{rem}
The paper \cite{Bar15} also contains a detailed discussion on the nature of the spectral truncation that applies to the case at hand (see \cite[Section 6.3]{Bar15}). It depends on the decomposition of the Hilbert space of spinors into irreducible representations of $\text{Spin}(3)$. However, since we will not need the specific form of the truncation here, we refrain from including it here. 
\end{rem}

Let us now proceed to show that there is a $C^1$-approximate order isomorphism $(\breve\sigma,\sigma)$ 
for the sequence of spectral triples defined in \eqref{eq:spectral-triple-fuzzy-S2} and the spectral triple of \eqref{eq:spectral-triple-S2}. As a consequence, we thus re-establish part of the conclusion of \cite[Theorem 3.2]{Rie02} that the fuzzy sphere converges to the two-sphere in Gromov--Hausdorff distance as $n \to \infty$. Note that {\em loc.cit.} goes further in establishing that this limit is the unique limit of the sequence of fuzzy spheres (and also extends to more general coadjoint orbits). The reason we have included this example here is that it is formulated entirely in terms of spectral triples, and fits the general framework set up in Section \ref{sect:gh}.

\subsubsection{Berezin symbol and Berezin quantization}
Following \cite{Rie02} we start by defining maps $\sigma:L(V_n) \to C(S^2)$ and $\breve \sigma : C(S^2) \to L(V_n)$. Given a projection $P \in L(V_n)$, say, on the highest-weight vector of $V_n$, we define the {\em Berezin symbol} $\sigma: L(V_n) \to C(S^2)$ by \cite{Ber75}
\begin{equation}
\sigma(T)(g) \equiv \sigma_T(g) := \tr ( T \alpha_g (P)) 
\label{eq:berezin-symbol}
\end{equation}
where $\alpha_g$ is the action of $g \in SU(2)$ induced by conjugation on $L(V)$. Since $\alpha_u (P) = P$ for all $u \in U(1)$, it follows that $\sigma(T)$ is $U(1)$-invariant and thus descends to a function on $SU(2)/U(1)= S^2$. Moreover, we readily see that $\sigma$ is an $SU(2)$-equivariant map which will turn out to be useful later.

We let $\breve \sigma: C(S^2) \to L(V_n)$ be the adjoint of the map $\sigma$ when $C(S^2)$ comes equipped with the $L^2$-inner product and $L(V_n)$ with the Hilbert--Schmidt inner product. There is also the following explicit expression ({\em cf.} \cite[Sect.2]{Rie02}). 
\begin{prop}
\label{prop:berezin}
The map $\breve \sigma$ defined by $\breve \sigma(f) \equiv \breve \sigma_f= n \int f(g) \alpha_g(P) dg$ satisfies 
$$
\langle f, \sigma_T \rangle = \frac{1}{n} \tr (\breve \sigma_f T).
$$
Moreover, we may write the so-called {\em Berezin transform} as a convolution product
$$
\sigma(\breve \sigma_f)(g) = (f \ast H_P)(g) \equiv \int f(g h^{-1}) H_P(h) dh
$$
where $H_P$ is a probability measure defined by
$$
H_P(g) = n \tr (P \alpha_g(P)).
$$
\end{prop}
\proof
As in \cite[Sect.2]{Rie02} we check the formula for $\breve \sigma(f)$ by computing that 
$$
\langle f, \sigma_T \rangle = \int f(g) \tr( T \alpha_g(P)) dg =\tr \left( \int f(g)  T \alpha_g(P) dg  \right)
$$
so that the result follows. 

For the Berezin transform we then indeed have that
\begin{align*}
\sigma(\breve \sigma_f)(g) &= \tr \left( \breve \sigma_f \alpha_x (P) \right) 
= \tr \left (n \int f(h) \alpha_h(P) dh \alpha_g(P)\right)\\
& = n \int f(h) \tr (P  \alpha_{h^{-1}g}(P)) dh = n \int f(g h^{-1} )H_P(h)  dh
\end{align*}
using also that $H_P(h^{-1} ) = H_P(h)$. 
\endproof
Again, one readily observes that $\breve \sigma$ is an $SU(2)$-equivariant map.

\subsubsection{The sphere as a limit of matrix algebras}
We now show in a series of Lemma's that the conditions of Definition \ref{defn:GH-conv} hold for $R_n= \breve \sigma$ and $S_n = \sigma$. 

\begin{lma}
\label{lma:est-comm-f}
For any $f \in C^\infty(S^2)$ we have $\| \breve \sigma_f \| \leq \| f \|$ and $\| [D_n,\breve \sigma_f]\| \leq \| [D_{S^2},f]\|$. 
\end{lma}
\proof
The contractive property of $\breve \sigma$ is proved for instance in \cite[Theorem 1.3.5]{Lns98a} where  $\breve \sigma$ is the Berezin quantization map. Then, by $SU(2)$-equivariance of $\breve \sigma$ we have
$$
[D_n,\breve \sigma_f] = \sum_{j<k} \sigma^j \sigma^k  \otimes [L_{jk}, \breve \sigma_f ] = \left( \sum_{j<k} \sigma^j \sigma^k \otimes \breve \sigma([X_{jk},f]) \right) =  (1 \otimes \breve \sigma ) [D_{S^2},f].
$$
Since $\breve \sigma$ is a positive map from a commutative domain to a $C^*$-algebra, it follows by a Theorem by Stinespring \cite{Sti55} ({\em cf.} \cite[Theorem 3.11]{Pau02}) that $\breve \sigma$ is completely positive. But then, it follows from \cite[Proposition 3.6]{Pau02} that $\breve \sigma$ is completely bounded with $\| \breve\sigma \|_{cb} = \| \breve \sigma \|$. In particular, $\| 1 \otimes \breve \sigma \| \leq \| \breve \sigma\|\leq 1$ so that it follows that
\[
\| [D_n,\breve \sigma_f] \|  \leq \| (1 \otimes \breve \sigma )\| \|  [D_{S^2},f]\| \leq  \|  [D_{S^2},f]\|. \qedhere
\]
\endproof

\begin{lma}
There exists a sequence $\{\gamma_n\}$ converging to 0 such that
$$
\| f - \sigma (\breve \sigma_f) \| \leq \gamma_n \| [D_{S^2},f]\| 
$$
for all $f \in C^\infty(S^2)$.
\end{lma}
\proof
We exploit the expression for $\sigma (\breve \sigma_f)$ as a convolution product from Proposition \ref{prop:berezin}. Indeed, 
\begin{align*}
|f(g) -  \sigma (\breve \sigma_f) (g) |&= \left| \int (f(g)  - f(h) ) H_P (h^{-1} g) dh \right| \\
& \leq \| f \|_{\text{Lip}} \int d(g,h) H_P  (h^{-1} g) dh = \| f \|_{\text{Lip}} \int d(e,h) H_P  (h) dh 
\end{align*} 
where $d$ is the $SU(2)$-invariant (round) distance on $SU(2)/U(1)$ and $\| f \|_{\text{Lip}}$ is the corresponding Lipschitz seminorm of $f$. Since $\| f \|_{\text{Lip}} = \|[D_{S^2},f]\|$ by standard arguments \cite[Sect. VI.1]{C94} and $\int d(e,h) H_P  (h) dh \to 0$ as $n \to \infty$, the result follows. 
\endproof

\begin{lma}
\label{lma:est-comm-T}
For any $T \in L(V_n)$ we have $\| \sigma_T \| \leq \|T \|$ and $\| [D_{S^2},\sigma_T] \| \leq \| [D_n,T]\|$.
\end{lma}
\proof
The map $\sigma$ is a contraction: 
$$
\| \sigma_T\|  = \sup_g | \tr T \alpha_g(P) | \leq \| T \| \sup_g  \tr | \alpha_g(P)| = \| T\| .
$$
Since $\sigma$ is also $SU(2)$-equivariant we again find that 
$$
[D_{S^2},\sigma_T] = \sum_{j<k} \sigma^j \sigma^k  \otimes [X_{jk},  \sigma_T ] = \left( \sum_{j<k} \sigma^j \sigma^k \otimes  \sigma([L_{jk},T]) \right) =  (1 \otimes \sigma ) [D_n,T].
$$
Since the range of $\sigma$ is a commutative $C^*$-algebra it follows from \cite[Theorem 3.9]{Pau02} that $\| \sigma\|_{cb} = \| \sigma\|$. Hence 
$$
\| [D_{S^2},\sigma_T] \| \leq \| (1 \otimes \sigma )\| \|  [D_n,T]\| \leq \| [D_n,T]\|
$$
since $\sigma$ is a contraction. 
\endproof

\begin{lma}
\label{lma:est-T}
There exists a sequence $\{\gamma_n'\}$ converging to 0 such that
$$
\| T - \breve \sigma (\sigma_T) \| \leq \gamma'_n \| [D_n,T]\| 
$$
for all $T \in L(V_n)$.
\end{lma}
\begin{proof}
This is based on a highly non-trivial result \cite[Theorem 6.1]{Rie02} which states that there exists a sequence $\{\gamma_n'\}$ converging to 0 such that
$$
\| T - \breve \sigma (\sigma_T) \| \leq \gamma'_n L_n (T)
$$
for all $T \in L(V_n)$, where $L_n$ is the Lipschitz norm on $L(V_n)$ defined by
$$
L_n(T) = \sup_{g \neq e} \frac{ \| \alpha_g (T) - T \| }{l(g)}
$$
for a length function $g$ on $SU(2)$ that induces the round metric on $S^2$. However, as in the proof of \cite[Theorem 3.1]{Rie98} we may estimate
$$
L_n(T) \leq \sup_{X \in su(2) }  \{ \| [X,T] \|: \| X \| \leq 1 \} 
$$
while the right-hand side can be bounded from above by $k \| [D,T] \|$ for some constant $k$ independent of $n$ (as in the display preceding \cite[Theorem 4.2]{Rie98}. 
\end{proof}

We have thus verified that the maps $(\breve \sigma,\sigma)$ between $\{ L(V_n) , \C^2 \otimes L(V_n) ,D_n)$ and $(C^\infty(S^2), \C^2\otimes L^2(S^2), D_{S^2})$ forms a $C^1$-approximate order isomorphism and we may conclude from Theorem \ref{thm:GH-conv} that
\begin{prop}
The sequence of state spaces $\{ (\cS(L(V_n) ),d_n)\}_n$ converges in Gromov--Hausdorff distance to $(\cS(C(S^2)),d)$. 
\end{prop}
\newcommand{\noopsort}[1]{}\def\cprime{$'$}


\begin{thebibliography}{10}

\bibitem{ACN92}
J.~R. Angelos, C.~C. Cowen, and S.~K. Narayan.
\newblock Triangular truncation and finding the norm of a {H}adamard
  multiplier.
\newblock {\em Linear Algebra Appl.} 170 (1992)  117--135.

\bibitem{Bar15}
J.~W. Barrett.
\newblock Matrix geometries and fuzzy spaces as finite spectral triples.
\newblock {\em J. Math. Phys.} 56 (2015)  082301, 25.

\bibitem{BG16}
J.~W. Barrett and L.~Glaser.
\newblock {Monte Carlo simulations of random non-commutative geometries}.
\newblock {\em J. Phys.} A49 (2016)  245001.

\bibitem{Ber19}
T.~Berendschot.
\newblock Truncated geometry.
\newblock Master's thesis, Radboud University Nijmegen, 2019.

\bibitem{Ber75}
F.~A. Berezin.
\newblock General concept of quantization.
\newblock {\em Comm. Math. Phys.} 40 (1975)  153--174.

\bibitem{BBI01}
D.~Burago, Y.~Burago, and S.~Ivanov.
\newblock {\em A course in metric geometry}, volume~33 of {\em Graduate Studies
  in Mathematics}.
\newblock American Mathematical Society, Providence, RI, 2001.

\bibitem{CE77}
M.~D. Choi and E.~G. Effros.
\newblock Injectivity and operator spaces.
\newblock {\em J. Functional Analysis} 24 (1977)  156--209.

\bibitem{C89}
A.~Connes.
\newblock Compact metric spaces, {F}redholm modules, and hyperfiniteness.
\newblock {\em Ergodic Theory Dynam. Systems} 9 (1989)  207--220.

\bibitem{C94}
A.~Connes.
\newblock {\em Noncommutative Geometry}.
\newblock Academic Press, San Diego, 1994.

\bibitem{CS20}
A.~Connes and W.~van Suijlekom.
\newblock Spectral truncations in noncommutative geometry and operator systems.
\newblock To appear in {\em Commun. Math. Phys.} [arXiv:2004.14115].

\bibitem{ALM14}
F.~D'Andrea, F.~Lizzi, and P.~Martinetti.
\newblock Spectral geometry with a cut-off: topological and metric aspects.
\newblock {\em J. Geom. Phys.} 82 (2014)  18--45.

\bibitem{ALV13}
F.~D'Andrea, F.~Lizzi, and J.~C. V\'{a}rilly.
\newblock Metric properties of the fuzzy sphere.
\newblock {\em Lett. Math. Phys.} 103 (2013)  183--205.

\bibitem{GS19b}
L.~Glaser and A.~Stern.
\newblock Reconstructing manifolds from truncated spectral triples.
\newblock arXiv:1912.09227.

\bibitem{GS19a}
L.~Glaser and A.~Stern.
\newblock {Understanding truncated non-commutative geometries through computer
  simulations}.
\newblock {\em J. Math. Phys.} 61 (2020)  033507.

\bibitem{GP95}
H.~Grosse and P.~Pre\v{s}najder.
\newblock The {D}irac operator on the fuzzy sphere.
\newblock {\em Lett. Math. Phys.} 33 (1995)  171--181.

\bibitem{Kaa20}
J.~Kaad.
\newblock Exterior products of compact quantum metric spaces.
\newblock (work in progress).

\bibitem{Ker03}
D.~Kerr.
\newblock Matricial quantum {G}romov-{H}ausdorff distance.
\newblock {\em J. Funct. Anal.} 205 (2003)  132--167.

\bibitem{KL09}
D.~Kerr and H.~Li.
\newblock On {G}romov-{H}ausdorff convergence for operator metric spaces.
\newblock {\em J. Operator Theory} 62 (2009)  83--109.

\bibitem{Lns98a}
N.~P. Landsman.
\newblock {\em Mathematical Topics between Classical and Quantum Mechanics}.
\newblock Springer, New York, 1998.

\bibitem{Lat15}
F.~Latr\'{e}moli\`ere.
\newblock Convergence of fuzzy tori and quantum tori for the quantum
  {G}romov-{H}ausdorff propinquity: an explicit approach.
\newblock {\em M\"{u}nster J. Math.} 8 (2015)  57--98.

\bibitem{Lat15b}
F.~Latr\'{e}moli\`ere.
\newblock The dual {G}romov-{H}ausdorff propinquity.
\newblock {\em J. Math. Pures Appl. (9)} 103 (2015)  303--351.

\bibitem{Lat16}
F.~Latr\'emoli\`ere.
\newblock The quantum {G}romov-{H}ausdorff propinquity.
\newblock {\em Trans. Amer. Math. Soc.} 368 (2016)  365--411.

\bibitem{Lat16b}
F.~Latr\'{e}moli\`ere.
\newblock Quantum metric spaces and the {G}romov-{H}ausdorff propinquity.
\newblock In {\em Noncommutative geometry and optimal transport}, volume 676 of
  {\em Contemp. Math.}, pages 47--133. Amer. Math. Soc., Providence, RI, 2016.

\bibitem{Mad92}
J.~Madore.
\newblock The fuzzy sphere.
\newblock {\em Classical Quantum Gravity} 9 (1992)  69--87.

\bibitem{Pau02}
V.~Paulsen.
\newblock {\em Completely bounded maps and operator algebras}, volume~78 of
  {\em Cambridge Studies in Advanced Mathematics}.
\newblock Cambridge University Press, Cambridge, 2002.

\bibitem{Rie98}
M.~A. Rieffel.
\newblock Metrics on states from actions of compact groups.
\newblock {\em Doc. Math.} 3 (1998)  215--229.

\bibitem{Rie99}
M.~A. Rieffel.
\newblock Metrics on state spaces.
\newblock {\em Doc. Math.} 4 (1999)  559--600.

\bibitem{Rie00}
M.~A. Rieffel.
\newblock Gromov-{H}ausdorff distance for quantum metric spaces.
\newblock {\em Mem. Amer. Math. Soc.} 168 (2004)  1--65.

\bibitem{Rie02}
M.~A. Rieffel.
\newblock Matrix algebras converge to the sphere for quantum
  {G}romov-{H}ausdorff distance.
\newblock {\em Mem. Amer. Math. Soc.} 168 (2004)  67--91.

\bibitem{SS03}
E.~M. Stein and R.~Shakarchi.
\newblock {\em Fourier analysis}, volume~1 of {\em Princeton Lectures in
  Analysis}.
\newblock Princeton University Press, Princeton, NJ, 2003.
\newblock An introduction.

\bibitem{Sti55}
W.~F. Stinespring.
\newblock Positive functions on {$C^*$}-algebras.
\newblock {\em Proc. Amer. Math. Soc.} 6 (1955)  211--216.

\bibitem{Tra92}
A.~Trautman.
\newblock Spinors and the {D}irac operator on hypersurfaces. {I}. {G}eneral
  theory.
\newblock {\em J. Math. Phys.} 33 (1992)  4011--4019.

\end{thebibliography}
\end{document}